\documentclass[10pt]{amsart}
\usepackage[margin=1in]{geometry}
\usepackage{amsmath,amssymb,setspace}
\usepackage{indentfirst}
\usepackage[small,nohug,heads=vee]{diagrams}
\diagramstyle[labelstyle=\scriptstyle]

\newtheorem{theorem}{Theorem}
\newtheorem{corollary}[theorem]{Corollary}

\newtheorem{proposition}[theorem]{Proposition}
\newtheorem{definition}[theorem]{Definition}

\DeclareMathOperator{\aut}{Aut}
\DeclareMathOperator{\inn}{Inn}
\DeclareMathOperator{\res}{Res}
\DeclareMathOperator{\ind}{Ind}

\DeclareMathOperator{\ints}{\mathbb{Z}}
\DeclareMathOperator{\id}{\text{id}}
\DeclareMathOperator{\irr}{\text{Irr}}
\DeclareMathOperator{\rep}{\bold{Rep}}

\begin{document}

\title{Hopf Modules and Representations of Finite Wreath Products}\author{Seth Shelley-Abrahamson}\date{August 2015}\maketitle

\begin{abstract} For a finite group $G$ one may consider the associated tower $S_n[G] := S_n \ltimes G^n$ of wreath product groups.  In \cite{Zel} Zelevinsky associates to such a tower a positive self-adjoint Hopf algebra (PSH-algebra) $R(G)$ as the infinite direct sum of the Grothendieck groups $K_0(\bold{Rep}-S_n[G])$.  In this paper, we study the interaction via induction and restriction of the PSH-algebras $R(G)$ and $R(H)$ associated to finite groups $H \subset G$.  A class of Hopf modules over PSH-algebras with a compatibility between the comultiplication and multiplication involving the Hopf $k^{th}$-power map arise naturally and are studied independently.  We also give an explicit formula for the natural PSH-algebra morphisms $R(H) \rightarrow R(G)$ and $R(G) \rightarrow R(H)$ arising from induction and restriction.  In an appendix, we consider the more general family of \emph{restricted wreath products} which are subgroups of wreath product groups analogous to the subgroups $G(m, p, n)$ of the wreath product cyclotomic complex reflection groups $G(m, 1, n)$.

\end{abstract}

\tableofcontents

\section{Introduction} In \cite{Zel} Zelevinsky introduced \emph{positive self-adjoint Hopf algebras} (PSH-algebras), a certain class of Hopf algebras with additional structure.  He showed that the axioms defining these objects are sufficiently rigid to permit a simple classification theorem, and this classification was subsequently used to study the complex representations of symmetric groups, finite wreath product groups $S_n \ltimes G^n$, and finite general linear groups.  In each of these cases, an associated PSH-algebra is constructed as an infinite direct sum of the relevant Grothendieck groups.  The multiplication is related to induction of representations, and the comultiplication is related to restriction.  For instance, for the symmetric groups the product of a class $[\sigma]$ of a representation of $S_k$ with a class $[\tau]$ of a representation of $S_l$ is given by taking the class in $K_0$ of the induced representation $\ind_{S_k \times S_l}^{S_{k + l}} (\sigma \otimes \tau)$.  The construction for finite wreath products is entirely analogous, and for the finite general linear groups the multiplication is given by parabolic induction.  The PSH-algebra axioms then encode a collection of essential properties of the representations, including Frobenius reciprocity and Mackey's theorem on the composition of the induction and restriction functors.  In the case of the symmetric groups, the PSH-algebra arising will be denoted $R$, and is isomorphic as a PSH-algebra to the Hopf algebra of integral symmetric functions.

In addition to considering PSH-algebras themselves, it is equally natural to consider their modules, and this perspective is made abundantly clear in van Leeuwen's study \cite{vanL} of finite symplectic and orthogonal groups.  In the case of the finite general linear groups, the Levi decomposition of parabolic subgroups gives rise naturally to a certain induction functor, parabolic induction, which associates to a pair of representations $\sigma$ and $\tau$ of the groups $GL_n(\mathbb{F}_q)$ and $GL_m(\mathbb{F}_q)$, respectively, a representation of the group $GL_{n + m}(\mathbb{F}_q)$.  Dual to this process one can construct parabolic restriction functors, and these functors underly the multiplication and comultiplication in the PSH-algebra associated to these finite general linear groups.  By contrast, the Levi decomposition for parabolic subgroups of the finite symplectic groups, for example, gives rise not to an algebra structure on the direct sum of Grothendieck groups but rather to a module/comodule structure over the PSH-algebra associated again to the finite general linear groups.  This Hopf module carries a strong compatibility between the action and coaction map which should be seen as a twisted version of the usual Hopf axiom for modules.

At the level of the Weyl groups, van Leeuwen's construction suggests that there should be a corresponding structure of twisted Hopf module over the PSH-algebra associated to the symmetric groups $S_n$, the Weyl groups of the finite general linear groups, on the direct sum of the Grothendieck groups of the type B/C Weyl groups $S_n[\ints/2\ints]$.  More generally, in the spirit of Zelevinsky, one may consider arbitrary finite wreath products.  More specifically, if $G$ is a finite group, one may consider the tower of wreath products $S_n[G] := S_n \rtimes G^n$.  Associated to this tower is a PSH-algebra $R(G)$, as shown in \cite[\S 7]{Zel}.  When $H \subset G$ is a subgroup, we will give $R(G)$ the natural structure of a Hopf module over $R(H)$.  Under certain conditions on $H$ and $G$, we will see that this Hopf module has a strong compatibility between the multiplication and comultiplication, in which the Hopf $[G : H]^{th}$-power map on $R(H)$ plays a key role.

We axiomatize the properties of these modules by defining ``$k$-PSH modules'' over any PSH-algebra, where $k \geq 0$ is an integer.  For example, in the setting just mentioned, we will show that $R(G)$ is a $[G : H]$-PSH module over $R(H)$.  We will prove a direct sum decomposition theorem for $k$-PSH modules in analogy with the tensor product decomposition theorem of Zelevinsky for PSH-algebras \cite[\S 2]{Zel} and with the direct sum decomposition theorem of van Leeuwen for twisted Hopf modules \cite{vanL}.  In the case of the $|G|$-PSH module $R(G)$ over the primitive PSH-algebra $R = R(1)$, we will give an explicit tensor product decomposition of $R(G)$ in analogy with the Wedderburn decomposition of $\mathbb{C} G$.

The structure of these modules essentially comes from the natural PSH-algebra morphisms $R(H) \rightarrow R(G)$ and $R(G) \rightarrow R(H)$ arising from induction and restriction, respectively.  For certain subgroups $H \subset G$, these morphisms have particularly nice structure, which is seen at the level of the Hopf modules in the form of the $k$-PSH property.  However, we are still able to obtain an interesting and explicit formula for these morphisms for arbitrary finite groups $H \subset G$.  This formula is given in terms of a ``matrix of Hopf powers'' corresponding to the matrix describing the induction and restriction of representations between the groups $G$ and $H$.  This formula generalizes in a natural way the formula for these maps given by Zelevinsky in the special case that $G$ is abelian and $H$ is trivial, in which case they are realized as iterated multiplication or comultiplication.

\section{Acknowledgements}

This paper represents in condensed form some of the results of my undergraduate thesis at Stanford University under the direction of Daniel Bump, to whom I would like to express my deepest gratitude.

\section{PSH-algebra Review: Definitions and the Decomposition Theorem}

Let $A$ be a commutative ring, and let $(H, \mu, \mu^*, e, e^*, T)$ be a Hopf algebra over $A$ with multiplication $\mu : H \otimes H \rightarrow H$, comultiplication $\mu^* : H \rightarrow H \otimes H$, unit $e : A \rightarrow H$, counit $e^* : H \rightarrow A$, and antipode $T : H \rightarrow H$.  We will refer to the Hopf algebra axiom that $\mu^*$ is a morphism of algebras as the \emph{Hopf axiom}.  We will use \emph{sumless Sweedler notation} to express the the comultiplication, writing $\mu^*(h) = h_{(1)} \otimes h_{(2)}$ with a sum of several simple tensors implied.

We will say the Hopf algebra $H$ is \emph{graded} if $H$ has an $A$-module grading $H = \bigoplus_{n \geq 0} H_n$ and if the maps $\mu, \mu^*, e, e^*, T$ are graded, where we give $H \otimes H$ a grading by $(H \otimes H)_n = \bigoplus_{k + l = n} H_k \otimes H_l$) and where we view $A$ as a graded $A$-module concentrated in degree 0.

We say a graded Hopf algebra is \emph{connected} if $e \colon A \rightarrow H_0$ and $e^* \colon H_0 \rightarrow A$ are mutually inverse isomorphisms, where $H_0$ is the degree 0 component of $H$.

A \emph{trivialized group} (T-group) is a free abelian group $M$ with a specified $\ints$-basis $\Omega = \Omega(M) \subset M$.  The elements of $\Omega$ are called the \emph{irreducible} elements of $M$.  This structure induces a positive-definite symmetric bilinear form $\langle \cdot, \cdot \rangle \colon M \times M \rightarrow \ints$ by declaring $\Omega$ an orthonormal basis.  We may then define the \emph{positive} elements of $M$ to be the elements $x \in M$ with $\langle x, \omega \rangle \geq 0$ for all $\omega \in \Omega$; these are the elements with all nonnegative coefficients when expressed as a sum of elements of $\Omega$.  If $\langle \omega, x \rangle > 0$ then $\omega$ is called an \emph{irreducible constituent} of $x$, and sometimes we say $x$ ``contains'' $\omega$ - this is written as $\omega \leq x$.  Direct sums of T-groups are given the structure of $T$-groups by taking as the irreducible elements the disjoint unions of the irreducible elements of the direct summands.  A T-group is said to be \emph{graded} if it is graded as an $A$-module and if the irreducible elements are homogeneous with respect to that grading.  Tensor products of T-groups are given the structure of T-groups in which the irreducible elements are the tensor products of irreducible elements of the tensor factors.  A \emph{positive map} of T-groups sends positive elements to positive elements. If $M$ and $N$ are $T$-groups and $f \colon M \rightarrow N$, $g \colon N \rightarrow M$ are abelian group morphisms, then we say $f$ and $g$ are \emph{adjoint} (also say $f$ is adjoint to $g$) if $\langle f(m), n \rangle_N = \langle m, g(n) \rangle_M$ for all $m \in M$ and $n \in N$. As usual, we say $f$ is \emph{self-adjoint} is $f$ is adjoint to itself.  We will take $\ints$ to have the canonical T-group structure with distinguished generator 1.

\begin{definition} A \emph{positive self-adjoint Hopf algebra} (PSH-algebra) is a connected Hopf algebra $(H, \mu, \mu^*, e, e^*, T)$ over $\ints$ which is also a graded $T$-group such that the maps $\mu, \mu^*, e, e^*$ are positive and graded, and such that the pairs $(\mu, \mu^*)$ and $(e, e^*)$ are pairs of adjoint maps.\end{definition}

The most fundamental example of a PSH-algebra is the PSH-algebra $R$ mentioned in the introduction constructed from representations of the symmetric group.  Let $R_0 = \ints$ and, for $n > 0$, let $R_n := K_0(\bold{Rep}-S_n)$ be the Grothendieck group of the category of finite-dimensional complex representations of $S_n$, and let $R$ be the graded direct sum: $$R = \bigoplus_{n \geq 0} R_n.$$   Give $R$ the graded T-group structure in which $\Omega(R_0) = \{1\}$ and, for $n > 0$, $\Omega(R_n)$ consists of the isomorphism classes of the irreducible representations of $S_n$.  In view of the natural embedding of $S_k \times S_l$ in $S_{k + l}$ and the identification of irreducible representations of $S_k \times S_l$ with tensor products of irreducible representations of $S_k$ and $S_l$, the multiplication $\mu \colon R \otimes R \rightarrow R$ is defined on positive simple tensors $[\sigma \otimes \tau] \in R_k \times R_l$ by $\mu([\sigma \otimes \tau]) = [\ind_{S_k \times S_l}^{S_{k + l}} \sigma \otimes \tau]$ and extended by linearity to $R \otimes R$.  The comultiplication $\mu^* \colon R \rightarrow R \otimes R$ is defined on $R_n$ similarly by setting $\mu^*([\sigma]) = \sum_{k + l = n} \res_{S_k \times S_l}^{S_n} [\sigma]$ for representations $\sigma$ of $S_n$.  The unit $e \colon \ints \rightarrow R_0 = \ints$ is the identity, and the counit is the graded projection to $R_0$.  That $R$ is a PSH-algebra then encodes several facts about representations - the T-group structure differentiates the classes of honest representations from those of virtual representations, and positivity properties reflect that induction and restriction send representations to representations, self-adjointness reflects Frobenius reciprocity, and the Hopf axiom reflects Mackey's theorem on the composition of induction and restriction (see, for example, \cite[\S32]{Bump}).

$R$ is the fundamental, or ``universal,'' PSH-algebra from which all others can be constructed.  Recall that an element $h$ of a Hopf algebra $H$ is called \emph{primitive} if $\mu^*(h) = h \otimes 1 + 1 \otimes h$.  Suppose $H$ is any PSH-algebra.  Let $I = \bigoplus_{n > 0} H_n$ be the \emph{augmentation ideal} in $H$, and set $I^2 = \mu(I \otimes I)$.  By adjointness, the primitive elements of $H$ then have an alternative description as the elements of the subgroup $P$ of $I$ orthogonal to $I^2$, \cite[Lemma 1.7]{Zel}.  Given any two irreducible elements  $\omega, \omega' \in I$ (so of positive degree), we may ``generate'' additional irreducible elements by considering the irreducible constituents of the product $\omega \omega'$.  The primitive irreducible elements are then precisely the irreducible elements that are not generated in this fashion, and are the minimal collection of irreducibles from which all others can be obtaied.  For instance, for the PSH-algebra associated with the complex representation theory of the finite general linear groups, the primitive irreducible representations (the \emph{cuspidal representations}) are those which do not occur as irreducible constituents of representations obtained by (nontrivial) parabolic induction.  As another example, the class of the trivial representation of $S_1$ is the unique primitive irreducible element of the PSH-algebra $R$.  We are now in a position to recall Zelevinsky's decomposition and classification theorems for PSH-algebras.  The tensor product of PSH-algebras is given the structure of a PSH-algebra in the obvious manner.

\begin{proposition} (Zelevinsky) If $H$ is a PSH-algebra with a single primitive irreducible element $\rho$, then for every $\omega \in \Omega(H)$ there exists $n \geq 0$ such that $\omega \leq \rho^n$, so $\deg(\rho) | \deg(\omega)$.  After a rescaling the grading so that $\deg(\rho) = 1$, and $H$ is isomorphic as a PSH-algebra to the PSH-algebra $R$ introduced above. \end{proposition}
\begin{proof} This is \cite[Theorem 3.1]{Zel}.\end{proof}

\begin{proposition} (Zelevinsky) Let $H$ be a PSH-algebra and let $\mathcal{C} \neq \emptyset$ be the primitive irreducible elements.  For any $\rho \in \mathcal{C}$, the T-subgroup $H(\rho)$ generated by the irreducible constituents of the $\rho^n$ for $n \geq 0$ is a PSH-subalgebra with the unique primitive irreducible element $\rho$, and the multiplication map $$\mu \colon \bigotimes_{\rho \in \mathcal{C}} H(\rho) \rightarrow H$$ is an isomorphism of PSH-algebras.\end{proposition}
\begin{proof} This is \cite[\S2.2]{Zel}.\end{proof}

In fact, $R$ is isomorphic as a PSH-algebra to the Hopf algebra of symmetric functions on countably many indeterminates.  A tremendous amount is known about symmetric functions, which in combination with the above theorems provides a very satisfactory description of PSH-algebras.

\section{The $k$-Hopf axiom and $k$-PSH Modules: Definitions}

This section introduces and establishes the basic properties of the objects of primary interest in this paper, positive self-adjoint $k$-Hopf modules ($k$-PSH modules).  These are a natural notion of module for PSH-algebras which involve the positivity and self-adjointness structure, the only unexpected aspect is the replacement of the Hopf axiom for Hopf algebras with a new compatibility relation, the ``$k$-Hopf'' axiom, in which the Hopf $k^{th}$-power map (iterated comultiplication to the $k^{th}$ tensor power followed by iterated multiplication) of the underlying PSH-algebra enters in an essential way.

Let $(H, \mu, \mu^*, e, e^*, T)$ be a Hopf algebra over the commutative ring $A$.  A \emph{Hopf module over $H$} is an $A$-module $M$ along with $A$-linear maps $\alpha \colon H \otimes M \rightarrow M$ (action), $\alpha^* \colon M \rightarrow H \otimes M$ (coaction) with the axioms of associativity, coassociativity, unit, and counit.  As for Hopf algebras, we will often use juxtaposition to represent the action and sumless Sweedler notation to represent the coaction.  Associativity thus states $a(bm) = (ab)m$ for all $a, b \in H$ and $m \in M$ (equivalently $\alpha \circ (\mu \otimes 1) = \alpha \circ (1 \otimes \alpha)$), while coassociativity states $(\mu^* \otimes 1) \circ \alpha^* = (1 \otimes \alpha^*) \circ \alpha^*$.  The unit axiom states that for $a \in A, m \in M$ we have $e(a)m = am$, while the counit axiom gives $e^*(m_{(1)})m_{(2)} = m$.

We say $M$ is \emph{graded} if $M$ has a grading $M = \bigoplus_{n \geq 0} M_n$ as an $A$-module, if $H$ is a graded Hopf algebra, and if the action and coaction are graded with respect to the grading $(H \otimes M)_n = \bigoplus_{k + l = n} H_k \otimes M_l$.  We say $M$ is \emph{positive} if $A = \ints$ and if both $M$ and $H$ are T-groups so that the action and coaction maps are positive.  If $M$ is both graded and positive we will always assume further that these are compatible in the sense that the irreducible elements are homogeneous.  We say $M$ is \emph{self-adjoint} if $M$ and $H$ are $T$-groups so that the action and coaction maps are mutually adjoint.

For $k > 0$, let $\mu^{(k)} \colon \bigotimes^k H \rightarrow H$ be iteration of multiplication $k$ times - by associativity this is unambiguous.  Similarly, we have $\mu^{*(k)} \colon H \rightarrow \bigotimes^k H$ by iterating comultiplication.  For example, note that $\mu^{(1)} = \mu^{*(1)} = id_H$ and $\mu = \mu^{(2)}$.  The composition $\Psi^k \colon = \mu^{(k)} \circ \mu^{*(k)} \colon H \rightarrow H$ will be called the \emph{Hopf $k^{th}$-power map}.  Note $\Psi^1$ = id$_H$. For the case $k = 0$, set $\mu^{(0)} = e$ and $\mu^{*(0)} = e^*$ and set $\Psi^0 = e \circ e^*$.  We then have:

\begin{proposition} $\Psi^k$ is a morphism of algebras when $H$ is commutative, and a morphism of coalgebras when $H$ is cocommutative. If $H$ is a PSH-algebra, $\Psi^k$ is self-adjoint and a PSH-algebra morphism, i.e. $\Psi^k$ is a self-adjoint, positive, graded Hopf algebra morphism. For $k \geq 1$, $\Psi^k$ commutes with every Hopf algebra endomorphism of $H$.\end{proposition}
\begin{proof} Recall $\Psi^0 = e \circ e^* = \mu \circ (1 \otimes T) \circ \mu^*$ where $T$ is the antipode.  The map $\mu^*$ is an algebra morphism by the Hopf axiom, similarly $\mu$ is a coalgebra morphism, and $T$ is an antihomomorphism of both algebras and coalgebras.  Thus $T$ is a morphism of algebras when $H$ is commutative and a morphism of coalgebras when $H$ is cocommutative.  $\mu$ is a morphism of algebras when $H$ is commutative, and $\mu^*$ is a morphism of coalgebras when $H$ is cocommutative.  Thus $\Psi^k$ is a morphism of algebras when $H$ is commutative and a morphism of coalgebras when $H$ is cocommutative.  The statement in the PSH case follows from the positivity and adjointness axioms.  That $\Psi^k$ commutes with every Hopf algebra endomorphism $f \colon H \rightarrow H$ is simply $f \circ \Psi^k = f \circ \mu^{(k)} \circ \mu^{*(k)} = \mu^{(k)} \circ f^{\otimes k} \circ \mu^{*(k)} = \mu^{(k)} \circ \mu^{*(k)} \circ f = \Psi^k \circ f$.\end{proof}

\begin{proposition} If $H$ is commutative and cocommutative, $\Psi^k \circ \Psi^l = \Psi^{kl}$ and $\mu \circ (\Psi^k \otimes \Psi^l) \circ \mu^* = \Psi^{k + l}$. \end{proposition}
\begin{proof} By the previous proposition, $\Psi^k$ is a Hopf algebra morphism.  For $k, l > 0$, the identity $\Psi^k \circ \Psi^l = \Psi^{kl}$ follows immediately from either commutativity or cocommutativity and the second identity follows from associativity and coassociativity.  The previous proposition gives $\Psi^k \circ \Psi^0 = \Psi^0 \circ \Psi^k$, and $\Psi^k \circ \Psi^0 = \Psi^0$ because for any algebra morphism $f \colon H \rightarrow H$ we have $f(\Psi^0(x)) = f(e(e^*(x))) = f(e^*(x)1) = e^*(x)f(1) = \Psi^0(x)$.  By commutativity both sides of the second identity are symmetric in $k$ and $l$, so we need only treat $k = 0$.  The counit axiom gives $\mu \circ (\Psi^0 \otimes 1) \circ \mu^* = \id$, so we have $\mu \circ (\Psi^0 \otimes \Psi^l) \circ \mu^* = \mu \circ ((\Psi^0 \circ \Psi^l) \otimes \Psi^l) \circ \mu^* = \mu \circ (\Psi^0 \otimes 1) \circ \mu^* \circ \Psi^l = \Psi^l.$\end{proof}

Let M be a Hopf module over $H$, and let $\tau \colon H \otimes H \rightarrow H$ denote the $A$-linear transposition map $\tau(x \otimes y) = y \otimes x$.  Let $k \geq 0$ be an integer.  We say that $M$ is a $k$-\emph{Hopf module} if it satisfies the \emph{$k$-Hopf axiom}, meaning that the following diagram commutes: 
\begin{diagram}
H \otimes M & \rTo^{\mu^* \otimes \alpha^*} &H \otimes H \otimes H \otimes M&\rTo^{\Psi^k \otimes \tau \otimes 1}&H \otimes H \otimes H \otimes M\\
\dTo^{\alpha}&&&&\dTo^{\mu \otimes \alpha}\\
M&&\rTo^{\alpha^*}&&H \otimes M\end{diagram}  In sumless Sweedler notation, this means for all $h \in H$ and $m \in M$ we have $\alpha^*(hm) = \Psi^k(h_{(1)})m_{(1)} \otimes h_{(2)}m_{(2)}$.  Observe that if $k = 1$, $M = H$, $\alpha = \mu$, and $\alpha^* = \mu^*$ then the $k$-Hopf axiom is precisely the Hopf axiom.

\begin{definition} If $H$ is a PSH-algebra, a \emph{$k$-PSH module over $H$} is a graded, positive, self-adjoint, $k$-Hopf module over $H$.\end{definition}

\section{Primitive Elements and Constructions}

Throughout this section, let $H$ be a PSH-algebra. For a Hopf module $(M, \alpha, \alpha^*)$ over $H$, we say an element $m \in M$ is \emph{module primitive} (or just \emph{primitive} if the context is clear) if $\alpha^*(m) = 1 \otimes m$, in contrast with the definition for Hopf algebras.  The primitive elements then form a subgroup, which will be denoted $Q$.  If $I = \bigoplus_{n > 0} H_n$ is the augmentation ideal as in Section 2, the grading and connectivity of $H$ along with the grading, unit, and counit for $M$ imply that for any $m \in M$, $\alpha^*(m) = 1 \otimes m + \alpha^*_+(m)$ for $\alpha^*_+(m) \in I \otimes M$.  Thus, $m$ is primitive $\iff$ $\alpha^*_+(m) = 0$.  This definition of primitivity is then justified by the following proposition:

\begin{proposition} Let $H$ be a PSH-algebra and let $M$ be $k$-PSH over $H$.  Let $IM \subset M$ be the submodule $\alpha(I \otimes M)$.  Then $Q$ is the orthogonal complement of $IM$.\end{proposition}
\begin{proof} The expression $\alpha^*(m) = 1 \otimes m + \alpha^*_+(m)$ along with the self-adjointness and grading gives $\langle m, \alpha(x) \rangle = \langle \alpha^*_+(m), x \rangle$ for $m \in M$ and $x \in I \otimes M$, from which the proposition follows.\end{proof}

Thus, the primitive irreducible elements of a $k$-PSH module are those that cannot be ``generated'' as an irreducible constituent of a nontrivial product of other irreducibles.  Using the previous proposition, we obtain:

\begin{proposition} Associativity and coassociativity follow from the other axioms defining $k$-PSH modules.\end{proposition}
\begin{proof} By self-adjointness, coassociativity follows from associativity, for then we have $\langle x \otimes y \otimes m, (\alpha^* \otimes 1)\alpha^*n\rangle = \langle (xy)m, n \rangle = \langle x(ym), n \rangle = \langle x \otimes y \otimes m, (1 \otimes \alpha^*)\alpha^*n\rangle$ for all $x, y \in H$ and $m, n \in M$.  For associativity, it suffices to show $x(ym) - (xy)m = 0$ for $x \in H_a$, $y \in H_b$, and $m \in M_c$ by linearity.  By the unit axiom certainly this is true if either $a = 0$ or $b = 0$, so we may do the proof by induction on $a + b + c$ and suppose $a + b > 0$.  Then $x(ym) - (xy)m \in IM$, so by the preceding proposition it suffices to show $x(ym) - (xy)m$ is primitive.  By the $k$-Hopf axiom for $M$, the Hopf axiom for $H$ and that $\Psi^k$ is a morphism of algebras, we have $$\begin{array} {lcl} &&\alpha^*(x(ym) - (xy)m) \\&= & \Psi^k(x_{(1)})(ym)_{(1)} \otimes x_{(2)}(ym)_{(2)} - \Psi^k((xy)_{(1)})m_{(1)} \otimes (xy)_{(2)}m_{(2)}\\&=&\Psi^k(x_{(1)}y_{(1)})m_{(1)} \otimes x_{(2)}(y_{(2)}m_{(2)}) - \Psi^k(x_{(1)}y_{(1)})m_{(1)}\otimes (x_{(2)}y_{(2)})m_{(2)}.\end{array}$$  The first tensor factors in each term of the sum agree, so and the second tensor factors agree by the inductive hypothesis when the degree is less than $a + b + c$, and these terms cancel.  By the counit axiom and the fact that $\Psi^k(1) = 1$ the remaining terms give $1 \otimes (x(ym) - (xy)m)$, so indeed $x(ym) - (xy)m$ is primitive.\end{proof}

Next we will discuss constructions involving $k$-Hopf modules.  These constructions will give a convenient language to describe the structure of the $k$-PSH modules associated with finite wreath products. Let $H$ be a commutative, cocommutative  Hopf algebra over the commutative ring $A$, and let $(M, \alpha, \alpha^*), (N, \beta, \beta^*)$ be $k$-Hopf and $l$-Hopf modules (respectively) over $H$.  Define the $A$-linear maps $\gamma \colon H \otimes M \otimes N \rightarrow M \otimes N$ and $\gamma^* \colon M \otimes N \rightarrow H \otimes M \otimes N$ by the formulas $\gamma(h \otimes m \otimes n) = h_{(1)}m \otimes h_{(2)}n$ and $\gamma^*(m \otimes n) = m_{(1)}n_{(1)} \otimes m_{(2)} \otimes n_{(2)}$.  Then we have:

\begin{proposition} $(M \otimes N, \gamma, \gamma^*)$ is a $(k + l)$-Hopf module over $H$.  If all objects are of the PSH type, then $M \otimes N$ is a $(k + l)$-PSH module. The usual isomorphisms $M \otimes N \cong N \otimes M$ and $(M \otimes N) \otimes P \cong M \otimes (N \otimes P)$ respect all the various structures.\end{proposition}

\begin{proof}  It is immediate that $\gamma$ and $\gamma^*$ are graded and that the unit and counit axioms hold.  Associativity follows from associativity of $M$ and $N$ and the Hopf axiom for $H$, and coassociativity follows similarly.  We need only check the $(k + l)$-compatibility axiom.  Recall that this amounts to checking the equality $$\gamma^* \circ \gamma = (\mu \otimes \gamma) \circ (\Psi^{k + l} \otimes \tau \otimes 1) \circ (\mu^* \otimes \gamma^*).$$  This follows from the second identity in Proposition 5 and the respective axioms for $M$ and $N$.  The statement in the PSH case is clear when we write $$\gamma = (\alpha \otimes \beta) \circ (1 \otimes \tau \otimes 1) \circ (\mu^* \otimes 1 \otimes 1)$$ $$\gamma^* = (\mu \otimes 1 \otimes 1) \circ (1 \otimes \tau \otimes 1) \circ (\alpha^* \otimes \beta^*),$$ from which positivity and self-adjointness readily follow.  The final statement is obvious.\end{proof}

Let $K$ also be a (graded) commutative, cocommutative Hopf algebra over the commutative ring $A$ with multiplication $\nu$ and comultiplication $\nu^*$, and let $\delta \colon K \rightarrow H$ and $\delta^* \colon H \rightarrow K$ be (graded) Hopf algebra morphisms such that $\delta^* \circ \delta = \Psi^l \colon K \rightarrow K$ for some $l \geq 0$.  Then for a $k$-Hopf module $(M, \alpha, \alpha^*)$ over $H$ there are the $A$-linear maps $\alpha_K  = \alpha \circ (\delta \otimes 1) \colon K \otimes M \rightarrow M$ and $\alpha_K^* = (\delta^* \otimes 1) \circ \alpha^* \colon M \rightarrow K \otimes M$.  Then we have:

\begin{proposition} $(M, \alpha_K, \alpha_K^*)$ is a Hopf module over $K$ with the $kl$-Hopf property.  If all objects are of the PSH type and the maps $\delta, \delta^*$ are mutually adjoint PSH-algebra morphisms (i.e. positive, mutually adjoint Hopf algebra morphisms), then $(M, \alpha_K, \alpha_K^*)$ is a $kl$-PSH module over $K$.\end{proposition}

\begin{proof} The $kl$-Hopf axiom follows from the first identity in Proposition 5, and the remaining axioms are immediate.  In the PSH-case, the grading, positivity, and adjointness follow from corresponding assumptions and the symmetry of the formulas defining $\alpha_K$ and $\alpha_K^*$.\end{proof}

\begin{proposition} The direct sum of $k$-Hopf modules (or $k$-PSH modules) is again a $k$-Hopf module ($k$-PSH module).\end{proposition}
\begin{proof} Clear.\end{proof}

\section{Decomposition of $k$-PSH Modules} In this section we establish a direct sum decomposition for $k$-PSH modules into summands with exactly one module-primitive irreducible element, analogous to van Leeuwen's direct sum decomposition for twisted PSH-modules.  Throughout this section let $(H, \mu, \mu^*)$ be a PSH-algebra (not necessarily with unique primitive irreducible element) and let $(M, \alpha, \alpha^*)$ be a $k$-PSH module over $H$.  In \cite{Zel} Zelevinsky considered the linear maps $x^* \colon H \rightarrow H$ adjoint to left-multiplication by $x \in H$.  These maps were central in his study of PSH-algebras.  Following this approach, we introduce analogous maps for $M$.

\begin{proposition} For $x \in H$ there exists a unique linear map $\widetilde{x} \colon M \rightarrow M$ adjoint to the left-multiplication map $M \rightarrow M$, $m \mapsto xm$, and for $m \in M$ there exists a unique linear map $\widetilde{m} \colon M \rightarrow H$ adjoint to right-muliplication by $m$.  $\widetilde{x}$ is given by the composition $$\begin{diagram}M&\rTo^{\alpha^*}&H \otimes M&\rTo^{\langle x, \cdot \rangle \otimes 1}&\ints \otimes M \cong M\end{diagram}$$ and $\widetilde{m}$ is given by $$\begin{diagram} M&\rTo^{\alpha^*}&H \otimes M&\rTo^{1 \otimes \langle m, \cdot \rangle}&H \otimes \ints \cong H.\end{diagram}$$ These maps satisfy (setting $H_n = M_n = 0$ for $n < 0$): $$(1)\ x \in H_p \implies \widetilde{x}(M_q) \subset M_{q - p}, \ \ \ m \in M_p \implies \widetilde{m}(M_q) \subset M_{q - p}$$ $$(2)\ x, y \in H \implies \widetilde{x} \circ \widetilde{y} = \widetilde{yx} = \widetilde{xy} = \widetilde{y} \circ \widetilde{x}$$ $$x \in H, m \in M \implies \widetilde{xm} = x^* \circ \widetilde{m} = \widetilde{m} \circ \widetilde{x}$$ $$(3)\ \widetilde{x}(ym) = [\Psi^k(x_{(1)})]^*(y)\widetilde{x_{(2)}}(m)$$ $$(4)\ \widetilde{m}(xn) = \Psi^k[m_{(1)}^*(x)]\widetilde{m_{(2)}}(n)$$ $$(5)\ \widetilde{m}(xn) = \Psi^k(x_{(1)})\widetilde{\widetilde{x_{(2)}}(m)}(n).$$\end{proposition}

\begin{proof} We will prove property $(3)$, and the others can be treated similarly.  For $x, y \in H$ and $m, n \in M$ we have $$\begin{array} {lcl} \langle \widetilde{x}(ym), n \rangle &=& \langle y \otimes m, (\alpha^* \circ \alpha)(x \otimes n)\rangle\\&=& \langle y \otimes m, (\Psi^k(x_{(1)}) \otimes x_{(2)})\alpha^*(n)\rangle\\&=&\langle \alpha((\Psi^k(x_{(1)}) \otimes x_{(2)})^*(y \otimes m)), n\rangle\\&=&\langle[\Psi^k(x_{(1)})]^*(y)\widetilde{x_{(2)}}(m), n\rangle\end{array}$$ from which property $(3)$ follows by the non-degeneracy of $\langle \cdot, \cdot\rangle$.\end{proof}

\begin{proposition} As before, let $P \subset H$and $Q \subset M$ be the subgroups of primitive elements.  Let $(p_i)_{i = 1}^r$ and $(p_j')_{j = 1}^s$ be tuples of pairwise equal or orthogonal elements of $P$, let $m, n \in Q$ be equal or orthogonal primitive elements of $M$, and let $\pi = p_1 \cdots p_rm$ and $\pi' = p_1' \cdots p_s'n$.  Then $\langle \pi, \pi' \rangle = 0$ unless $m = n$, $r = s$, and the $p_i$ and $p_j'$ are equal up to rearrangement, in which case we have $$\langle \pi, \pi' \rangle = k^rn_1! \cdots n_v!\langle p_1, p_1\rangle\cdots\langle p_r, p_r\rangle\langle m, m\rangle,$$ where $n_i$ is the number of appearances of the $i^{th}$ distinct element in the list $p_1, \dots, p_r$. (The case $M = H$ with its canonical $1$-PSH module structure is treated in \cite[Proposition 2.3]{Zel}.)\end{proposition}

\begin{proof} For $p \in P$ we have $\mu^*(p) = 1 \otimes p + p \otimes 1$, so by $(3)$ of the previous proposition we have $\widetilde{p}(xm) = (\Psi^k(p))^*(x)\widetilde{1}(m) + (\Psi^k(1))^*(x)\widetilde{p}(m) = kp^*(x)m + x\widetilde{p}(m)$ for any $x \in H$ and $m \in M$.  By \cite[Proposition 1.9f]{Zel}, that $p$ is primitive implies that $p^*$ is a derivation of $H$.  Therefore, we calculate $$\begin{array} {lcl} \langle \pi, \pi' \rangle & = & \langle (\Pi_{i \leq r} p_i) m, (\Pi_{j \leq s} p_j') n \rangle \\ & = & \langle (\Pi_{2 \leq i \leq r} p_i) m, \widetilde{p_1}((\Pi_{j \leq s} p_j')n) \rangle \\ & = & \langle (\Pi_{2 \leq i \leq r} p_i)m, k(p_1)^*(\Pi_{j \leq s} p_j')n + (\Pi_{j \leq s} p_j')\widetilde{p_1}(n) \rangle \\ & = & \displaystyle\sum\limits_{1 \leq l \leq s} k \langle (\Pi_{2 \leq i \leq r} p_i)m, (\Pi_{j \leq s, j \neq l} p_j')p_1^*(p_l') n \rangle \\& & + \langle (\Pi_{2 \leq i \leq r} p_i) m, (\Pi_{j \leq s} p_j') \widetilde{p_1}(n) \rangle. \end{array}$$  By the definition of primitivity and the orthogonality hypotheses, we have $p_1^*(p_l')$ is $0$ if $p_1 \neq p_l'$ and $\langle p_1, p_1\rangle$ otherwise, while $\widetilde{p_1}(n) = 0$.  The proposition then follows by induction.\end{proof}

Let $\mathcal{C} = \Omega(H) \cap P$ be the set of primitive irreducible elements of $H$ and let $\mathcal{D} = \Omega(M) \cap Q$ be the set of (module) primitive irreducible elements of $M$.  As in \cite[\S2.5]{Zel}, let $S(\mathcal{C}, \ints^{\geq 0})$ denote the additive monoid of functions $\mathcal{C} \rightarrow \ints^{\geq 0}$ of finite support.  For $d \in \mathcal{D}$ and $\phi \in S(\mathcal{C}, \ints^{\geq 0})$, define $$\pi_\phi = \prod_{c \in C} c^{\phi(c)} \in H, \ \ \ \pi_{d, \phi} = \pi_\phi d \in M.$$  Let $\Omega{\phi}$ be the set of irreducible elements $\omega \in H$ such that $\omega \leq \pi_\phi$, and similarly let $\Omega{d, \phi}$ be the set of irreducible constituents of $\pi_{d, \phi}$ in $M$.  Finally, set $$H(\phi) = \bigoplus_{\omega \in \Omega(\phi)} \ints \omega, \ \ \ M(d, \phi) = \bigoplus_{\omega \in \Omega(d, \phi)} \ints\omega, \ \ \ M(d) = \bigoplus_{\phi \in S(\mathcal{C}, \ints^{\geq 0})} M(d, \phi).$$

\begin{theorem} For $d, d' \in \mathcal{D}$ and $\phi, \phi' \in S(\mathcal{C}, \ints^{\geq 0})$, $\Omega(d, \phi)$ and $\Omega(d', \phi')$ are disjoint unless $(d, \phi) = (d', \phi')$.  $M$ has the $T$-group decomposition $$M = \bigoplus_{d \in \mathcal{D}, \phi \in S(\mathcal{C}, \ints^{\geq 0})} M(d, \phi)$$ and is graded with respect to $S(\mathcal{C}, \ints^{\geq 0})$ in the sense $$\alpha(H(\phi') \otimes M(d, \phi'')) \subset M(d, \phi' + \phi'')$$ $$\alpha^*(M(d, \phi)) \subset \bigoplus_{\phi' + \phi'' = \phi} H(\phi') \otimes M(d, \phi'').$$  In particular $M(d)$ is a $k$-PSH submodule of $M$, and as $k$-PSH modules: $$M = \bigoplus_{d \in \mathcal{D}} M(d).$$  Thus $M$ has a unique decomposition as a direct sum of $k$-PSH submodules with a single primitive irreducible element.\end{theorem}

\begin{proof} The disjointness of $\Omega(d, \phi)$ and $\Omega(d', \phi')$ for $(d, \phi) \neq (d', \phi')$ follows from the preceding proposition.  Let $\omega \in \Omega(M)$.  For the T-group decomposition we need $\omega \in \pi_{d, \phi}$ for some $(d, \phi)$.  This is trivial if $\omega \in \mathcal{D}$, so since $M_0 \subset Q$ it suffices to consider $\omega \in M_n$ not primitive and with $n > 0$.  Thus $\omega$ is \emph{not} in the orthogonal complement of $IM$ by Proposition 7, so there exists $x \in I, m \in M$ with $\omega \leq xm$.  By positivity we may assume $x \in \Omega(H)$, $m \in \Omega(M)$, and $x \neq 1$.  From \cite[Proposition 2.5a]{Zel} we have $x \leq \pi_\phi'$ for some $\phi' \in S(\mathcal{C}, \ints^{\geq 0})$, and by induction on $\deg \omega$ we can assume $m \leq \pi_{d, \phi''}.$  But then by positivity we have $\omega \leq \pi_\phi'\pi_{d, \phi''} = \pi_{d, \phi' + \phi''}.$

The first grading statement follows immediately from positivity, and the second follows from positivity, adjointness, and the disjointness statement.  \end{proof}

\section{$k$-PSH Modules and Representations of Wreath Products}

In this section we show that $k$-PSH modules appear naturally in the complex representation theory of finite wreath products and we describe the structure of these modules.  First, we recall the PSH-algebra associated with the complex representations of finite wreath products and apply the PSH-algebra decomposition theorem.  For a finite group $G$ and $n > 0$, the wreath product $S_n[G] = S_n \rtimes G^n$ is the semi-direct product associated to the action of $S_n$ on $G^n$ by permutation of indices.  Set $R(G) = \bigoplus_{n \geq 0} R_n(G)$ with $R_0(G) = \ints$ and, for $n > 0$, $R_n(G)$ the Grothendieck group of the category of finite-dimensional complex representations of $S_n[G]$.  Zelevinsky showed that then $R(G)$ has the structure of a PSH-algebra with multiplication and comultiplication given by identical formulas to the ones defining the corresponding structures on $R = R(1)$ with $S_n[G]$ in place of $S_n$.  The irreducible elements are the isomorphism classes of irreducible representations, and the irreducible primitive elements are the classes of the irreducible representations of $G = S_1[G]$.

Let $G$ be a finite group, and let $H \subset G$ be a subgroup.  There exists the graded linear map $\alpha \colon R(H) \otimes R(G) \rightarrow R(G)$ determined by setting $$\alpha = \ind_{S_k[H] \times S_l[G]}^{S_{k + l}[G]}$$ on representations, using the obvious identification of $S_k[H] \times S_l[G]$ as a subgroup of $S_{k + l}[G]$. Similarly, we have a graded linear map $\alpha^* \colon R(G) \rightarrow R(H) \otimes R(G)$ defined by $$\alpha^* = \displaystyle\sum\limits_{k + l = n} \res_{S_k[H] \times S_l[G]}^{S_n[G]}$$ on representations.  In the case $H = G$, we recover multiplication and comultiplication in the PSH-algebra $R(G)$.  These maps are graded, positive, and respect the unit and counit, and they are mutually adjoint by Frobenius reciprocity, so $R(G)$ is a positive self-adjoint Hopf module over $R(H)$.  Using Mackey's double-coset formula, the following two propositions show that under an additional hypothesis on $H$ $R(G)$ is a $[G : H]$-PSH module over $R(H)$.  We then give an explicit description of the module $R(G)$ over $R(H)$ in terms of only the PSH-algebra structures on these objects and the multiplication and comultiplication in the universal PSH-algebra.

To simplify the notation, given a sequence of integers $(n_1, \dots, n_l)$ let $W_{(n_1, \dots, n_l)}$ denote the direct product $S_{n_1}[G] \times \cdots \times S_{n_l}[G]$ and similarly let $V_{(n_1, \dots, n_l)}$ denote $S_{n_1}[H] \times \cdots \times S_{n_l}[H]$.  For a convenient description of wreath products, consider $S_n[G]$ as the group of monomial invertible matrices with entries in $\ints[G]$ and with all nonzero entries in $G$, and let $I_n$ denote the $n \times n$ identity matrix.

\begin{proposition}  Let $H \unlhd G$ be a normal subgroup with the property that every inner automorphism of $G$ restricts to an inner automorphism of $H$.  Let $t = [G \colon H]$, and let $\{g_1, \dots, g_t \}$ be a complete set of representatives for the elements of the quotient group $G / H$.  Suppose $p + q = r + s = n$.  Then the double-coset space $(V_p \times W_q) \backslash W_n / (V_r \times W_s)$ has a complete set of representatives parametrized by tuples $(a_1, \dots, a_t, b, c, d)$ of nonnegative integers satisfying the conditions $$a_1 + \cdots + a_t + b = r, c + d = s, a_1 + \cdots + a_t + c = p, b + d = q,$$ where the tuple $(a_1, \dots, a_t, b, c, d)$ corresponds to the representative $$\begin{bmatrix} g_1I_{a_1} & 0 & \cdots & 0 & 0 & 0 & 0 \\ 0 & g_2 I_{a_2} & \cdots & 0 & 0 & 0 & 0 \\ \vdots & \vdots & \ddots & 0 & 0 & 0 & 0 \\ 0 & 0 & 0 & g_t I_{a_t} & 0 & 0 & 0 \\ 0 & 0 & 0 & 0 & 0 & I_c & 0 \\ 0 & 0 & 0 & 0 & I_b & 0 & 0 \\ 0 & 0 & 0 & 0 & 0 & 0 & I_d \end{bmatrix}.$$\end{proposition}

\begin{proof}  Clear.\end{proof}

\begin{theorem} If $H \unlhd G$ is a normal subgroup with the property that every inner automorphism of $G$ restricts to an inner automorphism of $H$, then $(R(G), \alpha, \alpha^*)$  is a $[G : H]$-PSH module over $R(H)$.\end{theorem}

\begin{proof}  Again set $t = [G : H]$.  By the linearity of the maps involved, it suffices to verify the $t$-Hopf property on $\pi \otimes \sigma$, where $\pi$ is a representation of $V_r$ and $\sigma$ is a representation of $W_s$.  Suppose $r + s = n$.  First we compute $\alpha^\ast(\alpha(\pi \otimes \sigma))$ using Mackey's Theorem.  Recall that the definitions give $$\begin{array} {lcl} \alpha^\ast(\alpha(\pi \otimes \sigma)) & = & \displaystyle\sum\limits_{p + q = n} \res^{W_n}_{V_p \times W_q} ( \ind_{V_r \times W_s}^{W_n} (\pi \otimes \sigma)) \end{array}.$$

Let $N$ be the representative of the double-coset $(V_p \times W_q) \backslash W_n / (V_r \times W_s)$ as in the previous proposition parameterized by the tuple $(a_1, \dots, a_t, b, c, d)$ subject to the same constraints.  Since $H$ is normal, for $h \in H$ we have $g_i h g_j^{-1} \in H \iff i = j$, from which  it follows that $$(N (V_r \times W_s) N^{-1}) \cap (V_p \times W_q) = V_{(a_1, \dots, a_t, c, b)} \times W_d$$ (note the transposition of $b$ and $c$).

Let $\rho$ be the representation of the group $V_{(a_1, \dots, a_t, c, b)} \times W_d$ given by $\rho(x) = (\pi \otimes \sigma)(N^{-1}xN)$.  With the hypothesis that conjugation of $g \in G$ on $H$ is an inner automorphism, we may choose the representatives $g_i$ of $G/H$ so that $g_i$ centralizes $H$, and thus by Mackey $\alpha^*(\alpha(\pi \otimes \sigma))$ is $$\displaystyle\sum\limits_{a_1 + \cdots a_t + b = r, c + d = s} \ind_{V_{(a_1, \dots, a_t, c, b)} \times W_d}^{V_p \times W_q} ((1 \otimes \tau \otimes 1)(\res^{V_r}_{V_{(a_1, \dots, a_t, b)}}(\pi) \otimes \res^{W_s}_{V_c \times W_d}(\sigma)))$$ where recall $\tau$ is the transposition map swapping the $V_b$ and $V_c$ factors, corresponding to the $I_b$ and $I_c$ identity matrices appearing in the chosen representative $g_i$ from the previous proposition.  In view of the identity $$\Psi^t(\gamma) = \displaystyle\sum\limits_{a_1 + \cdots + a_t = a} \ind_{V_{(a_1, \dots, a_t)}}^{V_a} (\res^{V_a}_{V_{(a_1, \dots, a_t)}} (\gamma))$$ for representations $\gamma$ of $V_a$, we obtain $$\alpha^* \circ \alpha = (\mu_H \otimes \alpha) \circ (\Psi^t \otimes \tau \otimes 1) \circ (\mu_H^* \otimes \alpha^*)$$ (where $\mu_H$ is the multiplication and $\mu_H^*$ the comultiplication in $R(H)$ as a PSH-algebra) as needed.\end{proof}

Note that the previous proposition reaffirms the Hopf axiom for $R(G)$ in the case $H = G$.  We will give an alternative proof that $R(G)$ is a $k$-PSH module over $R(H)$, and, as a usual consequence of solving the same problem in two ways, we will recover some additional information, including the standard fact that the sum of the squares of the degrees of the irreducible representations of $G$ equals $|G|$.

\begin{proposition} Let $\delta \colon R(H) \rightarrow R(G)$ be the map induced by induction of representations from $S_n[H]$ to $S_n[G]$, and let $\delta^* \colon R(G) \rightarrow R(H)$ be the map induced by restriction.  Then $\delta$ and $\delta^*$ are mutually adjoint PSH-algebra morphisms, i.e. mutually adjoint, positive Hopf algebra morphisms. (Here $H \subset G$ can be an arbitrary subgroup.  The case H = 1 is \cite[Proposition 7.10a]{Zel}.)\end{proposition}

\begin{proof} Adjointness follows from Frobenius reciprocity, and positivity reflects that $\delta$ and $\delta^*$ are derived from functors.  By adjointness, therefore, it suffices to show that $\delta^*$ is a Hopf algebra morphism, as then the same will follow for $\delta$.  The associativity of restriction implies $\delta^*$ is a coalgebra morphism.  That $\delta^*$ is an morphism of algebras follows immediately from an easy application of Mackey's theorem, noting that $S_n[H] \backslash S_n[G] / (S_p[G] \times S_q[G])$ is trivial. \end{proof}

Observe that $\delta$ is given by right-multiplication by $1 \in R(G)$.  As $\delta^*$ is adjoint to $\delta$, it is therefore given by the composition $$\begin{diagram} R(G) & \rTo^{\alpha^*} & R(H) \otimes R(G) & \rTo^{1 \otimes \langle \cdot, 1 \rangle}& R(H) \otimes \ints \cong R(H).\end{diagram}$$

\begin{proposition} $\delta^* \circ \delta = \Psi^{[G : H]}$ for any normal subgroup $H \unlhd G$ with the property that every inner automorphism of $G$ restricts to an inner automorphism of $H$.\end{proposition}
\begin{proof} This follows immediately from the diagram defining the $[G : H]$-Hopf axiom and from the grading.\end{proof}

By the associativity of induction and restriction, the action and coaction maps $\alpha$ and $\alpha^*$ have the descriptions $\alpha = \mu_G \circ (\delta \otimes 1)$ and $\alpha^* = (\delta^* \otimes 1) \circ \mu_G^*$, where $\mu_G$, $\mu_G^*$ are the multiplication and comultiplication in $R(G)$.  Therefore, the module structure of $R(G)$ over $R(H)$ is determined by the PSH-algebra structures on these objects as well as the maps $\delta$ and $\delta^*$ between them.

We next give an explicit description of the PSH-algebra morphisms $\delta \colon R(H) \rightarrow R(G)$ and $\delta^* \colon R(G) \rightarrow R(H)$ for a finite group $G$ and \emph{any} finite subgroup $H \subset G$.  Recalling that the irreducible primitive elements of $R(G)$ are the classes of the irreducible representations of $G = S_1[G]$, Zelevinsky's decomposition theorem (Proposition 3) states that there is a PSH-algebra isomorphism $\Phi \colon \bigotimes^{|\irr(G)|} R \rightarrow R(G)$.  In \cite[Proposition 7.3]{Zel}, this isomorphism and its inverse are given explicitly, as follows.  Let $(\rho, V) \in \irr(G)$ be an irreducible representation, and let $R(\rho) \subset R(G)$ be the PSH-subalgebra defined in the statement of Proposition 3.  For a representation $(\pi, W)$ of $S_n$, define the representation $\Phi_\rho(\pi)$ of $S_n[G]$ in the space $W \otimes \bigotimes^n V$ such that $$\Phi_\rho(\pi)(\sigma)(w \otimes v_1 \otimes \cdots \otimes v_n) = \pi(\sigma)(w) \otimes v_{\sigma^{-1}(1)} \otimes \cdots \otimes v_{\sigma^{-1}(n)}$$ for $\sigma \in S_n$ and $$\Phi_\rho(\pi)(g_1, \dots, g_n)(w \otimes v_1 \otimes \cdots \otimes v_n) = w \otimes g_1v_1 \otimes \cdots \otimes g_nv_n$$ for $(g_1, \dots, g_n) \in G^n$.  This construction induces a linear map $\Phi_\rho \colon R \rightarrow R(G)$ which is a PSH-algebra isomorphism onto its image $R(\rho) \subset R(G)$.  The adjoint map $\Phi_\rho^* \colon R(G) \rightarrow R$ is orthogonal projection onto $R(\rho)$ and is given on representations $\bar{\pi}$ of $S_n[G]$ by the formula $\Phi_\rho^*(\bar{\pi}) = \hom_{G^n}(\otimes^n \rho, \bar{\pi}),$ where the $S_n$-action on this hom-space is given by $(\sigma.A)(x) = \sigma.(A(\sigma^{-1}.x))$ and the $S_n$-action on $\otimes^n \rho$ is by permutation of the tensor factors.  Therefore, the maps $$\Phi_G = \mu_G^{(|\irr(G)|)} \circ \bigotimes_{\rho \in \irr(G)} \Phi_\rho \colon R^{\otimes |\irr{G}|} \rightarrow R(G)$$ $$\Phi_G^* = \left(\bigotimes_{\rho \in \irr(G)} \Phi_\rho^* \right) \circ \mu_G^{*(|\irr(G)|)} \colon R(G) \rightarrow R^{\otimes |\irr{G}|}$$ are mutually adjoint PSH-algebra isomorphisms.  Of course, there is a very similar description of $R(H)$.  It is in terms of these descriptions that we will give formulas for the maps $\delta \colon R(H) \rightarrow R(G)$ and $\delta^* \colon R(G) \rightarrow R(H)$.

We now introduce the PSH-algebra morphism $\Psi^M \colon R^{\otimes k} \rightarrow R^{\otimes l}$ for $M = (m_{ij})$ a $l \times k$ matrix with entries in $\ints^{\geq 0}$, very analogous to the way in which linear transformations of vector spaces are described by matrices.  Let $\mu_k$ be the multiplication on $R^{\otimes k}$, similarly for $R^{\otimes l}$ and for comultiplication.  Then $\mu_k^{*(l)}$ maps  $R^{\otimes k}$ into $(R^{\otimes k})^{\otimes l}$.   In a sort of ``vertical'' Sweedler notation, consider writing the $kl$ tensor factors in a $l \times k$ matrix read from left to right, row by row, top to bottom.  As the comultiplication in $R^{\otimes k}$ is component-wise, this amounts to comultiplying each component ``down.''  The map $\Psi^M$ is then obtained by comultiplying $R^{\otimes k}$ $l$-times, applying $\Psi^{m_{ij}}$ to the $ij$-tensor entry, then multiplying the rows to obtain an element of $R^{\otimes l}$.  For instance, if we have $$M = \begin{bmatrix} 1&2&3\\4&5&6\end{bmatrix}$$ then $\Psi^M \colon R^{\otimes 3} \rightarrow R^{\otimes 2}$ is the map given on simple tensors by $$\Psi^M(x \otimes y \otimes z) = \Psi^1(x_{(1)})\Psi^2(y_{(1)})\Psi^3(z_{(1)}) \otimes \Psi^4(x_{(2)})\Psi^5(y_{(2)})\Psi^6(z_{(2)}).$$

\begin{proposition} $\Psi^M \colon R^{\otimes k} \rightarrow R^{\otimes l}$ is a PSH-algebra morphism.  We have the relations $$(1) \ \ (\Psi^M)^* = \Psi^{M^T}$$ $$(2) \ \ \Psi^M \circ \Psi^N = \Psi^{MN}$$ $$(3) \ \ \mu_l^{(n)} \circ (\Psi^{M_1} \otimes \cdots \otimes \Psi^{M_n}) \circ \mu_k^{*(n)} = \Psi^{\sum M_i}$$ $$\Psi^{nI} = \mu_k^{(n)} \circ \mu^{*(n)}_k =: \Psi_k^n$$ where $M^T$ denotes the transpose, $N$ is any $k \times m$ matrix and $M_1, \dots, M_n$ are any $l \times k$ matrices with entries in $\ints^{\geq 0}$, $nI$ is $n \geq 0$ times the identity matrix, and $\Psi_k^n$ is the $n^{th}$-Hopf power map on $R^{\otimes k}$.\end{proposition}

\begin{proof} $\Psi^M$ is a composition of PSH-algebra morphisms so is a PSH-algebra morphism itself.  Property (1) follows from the adjointness axiom for PSH-algebras, and Properties (2), (3), and (4) follow from Proposition 5.\end{proof}

Given a finite group $G$ and any subgroup $H$, let $M_{H,G} = (m_{\pi\rho})_{\pi \in \irr(H), \rho \in \irr(G)}$ be the $|\irr(H)| \times |\irr(G)|$ matrix with $\pi\rho$-entry $m_{\pi\rho} = \langle \pi, \res_H^G \rho \rangle$, the multiplicity of $\pi$ in the restriction of $\rho$ to $H$.  Clearly $m_{\pi\rho} \in \ints^{\geq 0}$.  Then with the earlier identifications $R(G) \cong \bigotimes^{|\irr(G)|} R$, $R(H) \cong \bigotimes^{|\irr(H)|} R$, we have the following

\begin{theorem} The PSH-algebra morphisms $\delta \colon R(H) \rightarrow R(G)$ and $\delta^* \colon R(G) \rightarrow R(H)$ are given by $$\delta^* = \Psi^{M_{H,G}} \text{,} \ \ \delta = \Psi^{M_{H, G}^T}.$$\end{theorem}

\begin{proof} In view of Propositions 17 and 19(1) we need only verify the identity for $\delta^*$.  Since $\delta^*$ is an algebra morphism, we then need only check that for $\rho \in \irr(G)$ the composition $$\begin{diagram} R&\rTo^{\ \Phi_{\rho\ \ }}&R(G)&\rTo^{\ \delta^*\ \ }&R(H)&\rTo^{\ \Phi_H^*\ \ }&R^{\otimes |\irr(H)|}\end{diagram}$$ is the map $$\left(\bigotimes_{\pi \in \irr(H)} \Psi^{\langle \pi, \delta^*\rho\rangle} \right) \circ \mu^{*(|\irr(H)|)}.$$  But $\delta^*$ and $\Phi_\rho$ are coalgebra morphisms, so $$\begin{array} {lcl} \Phi_H^* \circ \delta^* \circ \Phi_\rho &=& \left(\bigotimes_{\pi \in \irr(H)} \Phi_{\pi}^* \right) \circ \mu_H^{*(|\irr(H)|)} \circ \delta^* \circ \Phi_\rho\\&=& \left(\bigotimes_{\pi \in \irr(H)} (\Phi_{\pi}^* \circ \delta^* \circ \Phi_\rho) \right) \circ \mu^{*(|\irr(H)|)}.\end{array}$$  Each of these maps are algebra morphisms, so it suffice to check that they agree on a set of algebra generators of $R$, and for which we may choose the classes $x_n$ of the trivial representations of $S_n$, so we need only check $$\Phi_\pi^* \circ \delta^* \circ \Phi_\rho(x_n) = \Psi^{\langle \pi, \delta^*\rho\rangle}(x_n).$$  We have $$\delta^* \circ \Phi_\rho(x_n) = \bigoplus_{\pi_1, \dots, \pi_n \in \irr(H)} \left(\prod_{i = 1}^n m_{\pi_i\rho}\right) \pi_1 \otimes \cdots \otimes \pi_n.$$  By the definition of $\Phi_\pi^*$ and Schur's lemma, it follows that $$\Phi_\pi^* \circ \delta^* \circ \Phi_\rho(x_n) \cong \hom_{H^n}(\pi^{\otimes n}, (m_{\pi\rho} \pi)^{\otimes n}) \cong (\mathbb{C}^{m_{\pi\rho}})^{\otimes n}$$ where $S_n$ acts naturally by permuting the tensor factors.

We now need only check that this representation is isomorphic to $\Psi^{m_{\pi\rho}}(x_n)$, and for that it suffices to compare the characters.  Let $\sigma \in S_n$.  The standard basis of simple tensors of $(\mathbb{C}^m)^{\otimes n}$ is permuted under the action of $S_n$, and therefore the character value at $\sigma$ is the number of sequences of indices $(i_1, \dots, i_n)$, $1 \leq i_j \leq m$, such that $(\sigma(i_1), \dots, \sigma(i_n)) = (i_1, \dots, i_n)$.  Clearly this is $m^{\text{cycles}(\sigma)}$, where $\text{cycles}(\sigma)$ is the number of cycles of $\sigma$.

By Proposition 18, we have $\Psi^m(x_n) \cong \res_{S_n}^{S_n[K]}(\ind_{S_n}^{S_n[K]} x_n)$ for any group $K$ of order $m$.  The standard formula for induced characters therefore gives that the value of this character at $\sigma \in S_n$ is equal to $$\displaystyle\sum\limits_{d \in K^n} 1_{S_n}(d\sigma d^{-1})$$ where $1_{S_n}$ is the indicator function of $S_n \subset S_n[K]$.  We have $d\sigma d^{-1} = \sigma (\sigma^{-1}d\sigma d^{-1})$ with $\sigma^{-1}d\sigma d^{-1} \in K^n$, so $d\sigma d^{-1} \in S_n \iff d \sigma d^{-1} = \sigma$.  But the $d$ which centralize $\sigma$ are precisely the $d$ with all coordinates indexed by the same cycle of $\sigma$ equal.  There are $m^{\text{cycles}(\sigma)}$ such $d \in K^n$, as needed.\end{proof}

Combining the results of Theorem 16, Proposition 19, and Theorem 20, we have the following theorem which recovers, in the case $H = 1$, the standard fact that the sum of the squares of the dimensions of the irreducible complex representations of a finite group $G$ is $|G|$:

\begin{theorem} If $G$ is a finite group and $H \unlhd G$ is a normal subgroup so that the natural map $G \rightarrow \aut(H)$ induced by conjugation factors through the inclusion $\inn(H) \rightarrow \aut(H)$, then with $M_{H,G}$ as defined earlier we have the following matrix equality: $$M_{H, G} M_{H, G}^T = [G : H]I$$ where $I$ is the identity matrix of size $|\irr(H)|$.  Note the case $H = 1$ is precisely the standard result mentioned above.\end{theorem}

\begin{proof} All that is left is the observation that if $M \neq M'$ then $\Psi^M \neq \Psi^{M'}$.\end{proof}

Finally, further paraphrasing these results in the case $H = 1$, we describe the $|G|$-PSH module $R(G)$ over $R = R(1)$ in terms of the constructions of Propositions 9 and 10.  For $d \geq 0$, we have $\Psi^d \colon R \rightarrow R$ is a self-adjoint PSH-algebra morphism and $\Psi^d \circ \Psi^d = \Psi^{d^2}$, so setting $\delta = \delta^* = \Psi^d$ in Proposition 10 $R$ can be given the structure of a $d^2$-PSH module over itself.  Let $R^{(d)}$ denote $R$ with this module structure.  Then as a corollary of Theorem 20 we have the following

\begin{corollary} For any finite group $G$, we have, as $|G|$-PSH modules over $R$, $$R(G) \cong \bigotimes_{\omega \in \irr(G)} R^{(\dim \omega)}.$$\end{corollary}

\section{Appendix: Restricted Wreath Products} Let $G$ be a finite abelian group.  We may then realize the wreath product $S_n[G] := S_n \rtimes G^n$ as the group of monomial matrices with all nonzero entries in $G$.  As $G$ is abelian there is a surjection $S_n[G] \rightarrow G$ by taking the sum of the elements in $G$ appearing in a matrix.  For a subgroup $H \subset G$ let $G_n(G, H)$ denote the kernel of the composition $S_n[G] \rightarrow G \rightarrow G/H$, so that $G_n(G, G) = S_n[G]$ and $G_n(G, H)$ is the group of monomial matrices with entries in $G$ whose entries sum to an element of $H$.  We will refer to the groups $G_n(G, H)$ as \emph{restricted wreath products}.

Note that when $G$ is cyclic this construction yields the finite complex reflection groups in the family $G(m, p, n)$ where $p$ divides $m$.  Specifically, we have $G(m, p, n) = G_n(\ints/m, p\ints/m)$.

Let $R_0(G, H) = \ints$ and for $n > 0$ let $R_n(G, H) = K_0(\bold{Rep}-G_n(G, H))$ denote the Grothendieck group of the category of finite dimensional complex representations of $G_n(G, H)$.  We then construct the graded abelian group $$R(G, H) = \bigoplus_{n \geq 0} R_n(G, H)$$ which has a $T$-group structure with the distinguished graded basis given by the isomorphism classes of irreducible representations (along with $1 \in \ints$ in degree 0) along with a graded nondegenerate symmetric bilinear form given by the usual pairing of representations.  This form will be denoted $\langle \cdot, \cdot \rangle$.  Note that the irreducible elements form a graded orthonormal basis for $R(G, H)$.

Using induction and restriction, one can place graded product and coproduct structures on $R(G, H)$.  In particular, we have an embedding of groups $G_k(G, H) \times G_l(G, H) \subset G_{k + l}(G, H)$ by the block-diagonal embedding of matrices, so we have an induction functor $$\ind : \rep(G_k(G, H) \times G_l(G, H)) \rightarrow \rep(G_{k + l}(G, H))$$ and a restriction functor $$\res: \rep(G_{k + l}(G, H)) \rightarrow \rep(G_k(G, H) \times G_l(G, H)).$$  These are exact functors, so we obtain maps at the level of Grothendieck groups: $$m_{k, l} : R_k(G, H) \otimes R_l(G, H) \rightarrow R_{k + l}(G, H)$$ $$m_{k, l}^*: R_{k + l}(G, H) \rightarrow R_k(G, H) \otimes R_l(G, H)$$ in view of the natural isomorphism $$R_k(G, H) \otimes R_l(G, H) \cong K_0(\rep(G_k(G, H) \times G_l(G, H))).$$  For $k = 0$ or $l = 0$, let $m_{k, l}$ and $m^*_{k, l}$ be the maps given by the natural isomorphism $R_k \otimes \ints \cong \ints \otimes R_k \cong R_k$.  Set $$m := m_{G, H} =  \sum_{k, l \geq 0} m_{k, l} : R(G, H) \otimes R(G, H) \rightarrow R(G, H)$$ and $$m^* := m^*_{G, H} = \sum_{k, l \geq 0} m_{k, l}^* : R(G, H) \rightarrow R(G, H) \otimes R(G, H).$$  It is immediate that $m_{G, H}$ gives $R(G, H)$ the structure of a graded commutative algebra with unit and that $m_{G, H}^*$ gives $R(G, H)$ the structure of a graded cocommutative coalgebra with counit.  Furthermore, by Frobenius reciprocity $m_{G, H}$ and $m^*_{G, H}$ are adjoint operators with respect to the inner product on $R(G, H)$ and the induced graded inner product on $R(G, H) \otimes R(G, H)$.  As for usual wreath products, $m_{G, H}$ and $m_{G, H}^*$ are positive maps with respect to the $T$-group structure.

Next, we will construct a natural positive injective map of algebras $$\Phi: \bigotimes_{l \in H^*} R(G/H, 1) \hookrightarrow R(G, H)$$ where $H^*$ is the group of linear characters of $H$, and we will see a weak form of surjectivity in the sense that every irreducible element in $R(G, H)$ occurs as a constituent of some element of the image of this map.  For $H = G$, the case of usual wreath products, $R(G/G, 1) = R(1, 1)$ is the Hopf algebra of integral symmetric functions and the injection above is the usual isomorphism of Hopf algebras known in that case from Proposition 3.

Let $\phi \colon G_n(G, H) \rightarrow G_n(G/H, 1)$ be the map given by reducing the matrix entries mod $H$.  This gives rise to an exact sequence $$0 \rightarrow H^n \rightarrow G_n(G, H) \rightarrow G_n(G/H, 1) \rightarrow 0$$ where the first map is the diagonal embedding.  We obtain an additive functor $\phi^* : \rep(G_n(G/H, 1)) \rightarrow \rep(G_n(G, H))$ by pullback, which gives rise to a graded operator $\phi^* \colon R(G/H, 1) \rightarrow R(G, H)$.  As $\phi$ is surjective this map sends distinct irreducibles to distinct irreducibles, and in particular is an embedding of graded $T$-groups.

For $l \in H^*$, let $l_n$ be the linear character of $G_n(G, H)$ obtained by pulling back the linear character $l$ of $H$ by the homomorphism $G_n(G, H) \rightarrow H$.  We have the exact functor $\tau_{l, n} \colon \rep(G_n(G, H)) \rightarrow \rep(G_n(G, H))$ by tensoring with $l_n$, giving rise to a positive graded automorphism $\tau_l := \bigoplus_{n \geq 0} \tau_{l, n}$ of $R(G, H)$.  These operators have several nice properties.  We see $\tau_l \circ \tau_{l'} = \tau_{ll'}$.  In view of the inner product on $R(G, H)$ in terms of characters, we see $\tau_l^* = \tau_{\bar{l}} = \tau_{l^{-1}} = \tau_l^{-1}$, so $\tau_l$ is an orthogonal operator.  It is clear that $\tau_l$ is a map of coalgebras, but then $\tau_l^* = \tau_{\bar{l}}$ is also a map of coalgebras, so since the form on $R(G, H)$ is nondegenerate and $m_{G, H}$ and $m_{G, H}^*$ are mutually adjoint we conclude $\tau_l$ is also a map of algebras.  In summary, the rule $l \mapsto \tau_l$ gives an action of $H^*$ on $R(G, H)$ by positive graded orthogonal algebra/coalgebra automorphisms.

For $l \in H^*$, set $\Phi_l = \tau_l \circ \phi^* \colon R(G/H, 1) \rightarrow R(G, H)$.  We then have

\begin{proposition} $\Phi_l$ is an injective algebra and coalgebra morphism sending irreducibles to irreducibles. \end{proposition}

\begin{proof} In view of the preceding comments about $\tau_l$ and $\phi^*$ we need only check that $\phi^*$ respects multiplication and comultiplication.  It is obvious that $\phi^*$ is a map of coalgebras, and to establish that it is a map of algebras we need to check that the diagram $$\begin{diagram} R(G/H, 1) \otimes R(G/H, 1) &\rTo^{\phi^* \otimes \phi^*}& R(G, H) \otimes R(G, H)\\\dTo^{m_{G/H, 1}}&&\dTo^{m_{G, H}}\\R(G/H, 1) &\rTo^{\phi^*}&R(G, H)\end{diagram}$$ commutes.  For this just note that $\phi$ induces a bijection on the coset spaces $G_{k + l}(G, H)/(G_k(G, H) \times G_l(G, H))$ and $G_{k + l}(G/H, 1)/(G_k(G/H, 1) \times G_l(G/H, 1))$, and the commutativity of the diagram then follows immediately from the Frobenius formula for induced characters. \end{proof}

\begin{proposition}  The sub-(co)algebra $\Phi_l(R(G/H, 1))$ of $R(G, H)$ has a graded basis whose degree $n$ part consists of the isomorphism classes of those irreducible representations $\pi$ of $G_n(G, H)$ whose restriction to $H^n$ contains the irreducible constituent $l^{\otimes n}$.\end{proposition}

\begin{proof} From the construction and the previous proposition, we need only check that any such $[\pi]$ is in the image of $\Phi_l$.  Note that the $l^{\otimes n}$-isotypic piece of $\pi|_{H^n}$ is actually a submodule for $G_n(G, H)$, so $\tau_l^{-1}\pi$ is an irreducible representation of $G_n(G, H)$ with trivial $H^n$-action, so has the structure of an irreducible $G_n(G/H, 1) = G_n(G, H)/H^n$-representation $\pi'$.  But then $\tau_{l}^{-1}\pi = \phi^*\pi'$ so $\pi = \Phi_l(\pi)$, as needed.\end{proof}

\begin{proposition}  The sub-(co)algebras $\Phi_l(R(G/H, 1))$ are pairwise orthogonal. \end{proposition}

\begin{proof} If $\pi$ and $\sigma$ are irreducible representations of $G_n(G, H)$ which are $l^{\otimes n}$-isotypic and $l'^{\otimes n}$-isotypic upon restriction to $H^n$, respectively, for some distinct linear characters $l \neq l'$ of $H$, then we have $$\langle \pi, \sigma \rangle_{G_n(G, H)} \leq \langle \pi, \sigma\rangle_{H^n} = \deg(\pi)\deg(\sigma)\langle l^{\otimes n}, l'^{\otimes n}\rangle_{H^n} = 0$$ so $\langle \pi, \sigma \rangle = 0$, and in view of the previous proposition and its proof, our claim follows.\end{proof}

Now for $l \in H^*$ define the graded operator $\Psi_l \colon R(G, H) \rightarrow R(G/H, 1)$ on the degree $n$ part by the operator associated to the exact functor $\Psi_l \colon \rep(G_n(G, H)) \rightarrow \rep(G_n(G/H, 1))$ defined by $\Psi_l(\pi) = \hom_{H^n}(l_n, \pi)$.  The $G_n(G/H, 1)$-action is given by $g.A \mapsto \tilde{g}A\tilde{g}^{-1}$ for $A \in \hom_{H^n}(l_n, \pi)$ and $\tilde{g} \in G_n(G, H)$ any lift of $g \in G_n(G/H, 1)$.  This map $g.A$ does not depend on the choice of lifting of $g$ because $A$ commutes with the action of $H^n$.  Clearly $g.A \in \hom_{H^n}(l^{\otimes n}, \pi)$ so $\Psi_l(\pi)$ is indeed a $G_n(G/H, 1)$-representation, and clearly $\Psi_l$ is an additive functor.

\begin{proposition} The functor $\Psi_l$ is left adjoint to $\Phi_l$, and in particular the operators $\Psi_l$ and $\Phi_l$ are adjoints on the level of the Grothendieck groups.  $\Psi_l$ is a homomorphism of (co)algebras.\end{proposition}

\begin{proof} The first statement follows from tensor-hom adjunction.  The second statement follows by nondegeneracy of the forms on the Grothendieck groups, adjointness, and the fact that $\Phi_l$ is a homomorphism of (co)algebras.\end{proof} 

It is clear that $\Psi_l \circ \Phi_l$ is naturally isomorphic to the identity functor and that $\Phi_l \circ \Psi_l$ is naturally isomorphic to the functor $I_l$ on $\rep(G_n(G, H))$ given by projection to the $l^{\otimes n}$-isotypic piece for the $H^n$-action (recall this is always a $G_n(G, H)$-subrepresentation).  At the level of Grothendieck groups, we obtain:

\begin{proposition} $\Psi_l \circ \Phi_l \colon R(G/H, 1) \rightarrow R(G/H, 1)$ is the identity, and $$\Phi_l \circ \Psi_l \colon R(G, H) \rightarrow R(G, H)$$ is orthogonal projection onto the sub-(co)algebra $\Phi_l(R(G/H, 1)).$\end{proposition}

We now define the map mentioned at the start of this section $$\Phi \colon \bigotimes_{l \in H^*} R(G/H, 1) \rightarrow R(G, H)$$ as the product of the maps $\Phi_l$.  Given an $|H^*|$-tuple $\lambda = (\lambda_l)_{l \in H^*}$ of nonnegative integers, let $l(\lambda)$ denote the number of nonzero parts.  Write $G_\lambda(G/H, 1) := \bigotimes_{l \in H^*} G_{\lambda_l}(G/H, 1).$  Given irreducible representations $\pi_l$ of $G_{\lambda_l}(G/H, 1)$, let $\pi_\lambda := \bigotimes_{l \in H^*} \pi_l \in \bold{Rep}-G_\lambda(G/H, 1)$ and also identify it with its class in $\bigotimes_{l \in H^*} R(G/H, 1).$  Let $\mu = (\mu_l)_{l \in H^*}$ be another such tuple, with the same sum of parts $n := \sum_l \mu_l = \sum_l \lambda_l$ as $\lambda$.  Let also $\sigma_l$ be an irreducible representation of $G_{\mu_l}(G/H, 1)$ for each $l \in H^*$.  Then we have:

\begin{theorem} $\Phi$ is a graded, positive, injective map of algebras, and we have $$\langle \Phi(\pi_\lambda), \Phi(\sigma_\mu)\rangle = \delta_{\pi_\lambda, \sigma_\mu} [G : H]^{l(\lambda) - 1}.$$  $\Phi$ is weakly surjective in the sense that every irreducible element of $R(G, H)$ occurs as a constituent of some element of the image of $\Phi$.\end{theorem}

\begin{proof} It follows from previous results that $\Phi$ is a positive graded map of algebras.  By positivity, showing the given inner product formula will imply injectivity, so we start there, which is just an application of Mackey's double coset formula.  In particular, we have $$\begin{array} {lcl}&&\langle \Phi(\pi_\lambda), \Phi(\sigma_\mu)\rangle\\&=&\langle \ind_{G_\lambda(G, H)}^{G_n(G, H)} \bigotimes_{l \in H^*} \Phi_l(\pi_l), \ind_{G_\mu(G, H)}^{G_n(G, H)}\bigotimes_{l \in H^*} \Phi_l(\sigma_l)\rangle_{G_n(G, H)}\\&=&\sum_{\gamma \in G_\lambda\backslash G_n/G_\mu} \langle \bigotimes_{l \in H^*} \pi_l, (\bigotimes_{l \in H^*} \sigma_l)^\gamma\rangle_{G_\lambda \cap\gamma G_\mu \gamma^{-1}} \end{array}$$ Noting $H^n \subset G_\lambda \cap \gamma G_\mu \gamma^{-1}$, we have the bound $$\left\langle \bigotimes_{l \in H^*} \pi_l, \left(\bigotimes_{l \in H^*} \sigma_l\right)^\gamma\right\rangle_{G_\lambda \cap\gamma G_\mu \gamma^{-1}}$$ $$\leq \left(\prod_{l \in H^*} \deg(\pi_l)\deg(\sigma_l)\right) \left\langle \bigotimes_{l \in H^*} l^{\otimes \lambda_l}, \left(\bigotimes_{l \in H^*} l^{\otimes \mu_l}\right)^\gamma\right\rangle_{H^n}$$  As $G^n$ centralizes $H^n$, twisting by $\gamma$ amounts to just twisting by some element of $S_n$, permuting the tensor factors, and hence the final inner product is $0$ unless both $\lambda = \mu$ and $\bar{\gamma} \in S_n$ stabilizes the blocks of $G_\lambda(G/H, 1)$.  In this case we can clearly take the double coset representative $\gamma$ to be diagonal, so to compute the original inner product we need only sum over a collection of diagonal double coset representatives.  From the definition of $\Phi_l$, we see that any diagonal element of $G_{\lambda_i}(G, G)$ centralizes $\Phi_l(\sigma_l)$, and we conclude that each term of the above Mackey sum involving a diagonal representative yields a contribution of $1$ to the sum, as in that case $G_\lambda \cap \gamma G_\mu \gamma^{-1} = G_\lambda$ and it is just an inner product of an irreducible representation of $G_\lambda$ with itself.  The number of classes of the double coset space which have a diagonal representative is clearly $[G: H]^{l(\lambda) - 1}$ - indeed they are formed by a choice of element of $G/H$ for each nonzero $\lambda_i \times \lambda_i$ block such that the entire sum is $1 \in G/H$.

For our weak form of surjectivity, let $\pi \in \rep(G_n(G, H))$ be an irreducible representation.  The action of $S_n$ on $\pi$ induces an action of $S_n$ on the set of $H^n$-isotypic pieces by permuting the tensor factors, so we conclude some there is a nonzero $H^n$-isotypic piece of $\pi|_{H^n}$ of type $l_1^{\otimes \lambda_1} \otimes \cdots \otimes l_{|H|}^{\otimes \lambda_{|H|}}$ where $l_1, ..., l_{|H|}$ is an ordering of the linear characters of $H$. Then, $\pi|_{G_\lambda}$ contains some irreducible representation $\pi_1 \otimes \cdots \otimes \pi_{|H|}$ whose restriction $\pi_i|_{H^{\lambda_i}}$ contains $l_i^{\otimes \lambda_i}$.  By Proposition 24, we then have $\pi_i = \Phi_{l_i}(\pi_i')$ for some $\pi_i'$.  But then by Frobenius reciprocity $\pi$ is an irreducible constituent of $\Phi(\pi_\lambda)$, as needed.\end{proof}

Writing $$\Phi = m_{G, H}^{(|H^*|)} \circ \bigotimes_{l \in H^*} \Phi_l : \bigotimes_{l \in H^*} R(G/H, 1) \rightarrow R(G, H)$$ we also have the adjoint map $$\Psi = \bigotimes_{l \in H^*} \Psi_l \circ m_{G, H}^{*(|H^*|)} : R(G, H) \rightarrow \bigotimes_{l \in H^*} R(G/H, 1)$$ where $m_{G, H}^{(|H^*|)}$ and $m_{G, H}^{*(|H^*|)}$ denote iterated multiplication/comultiplication.  The first part of the previous theorem can then be recast as

\begin{corollary} $\Psi$ is a positive, graded, surjective map of coalgebras.  No positive element lies in its kernel. \end{corollary}

Note that in the case of usual wreath products, i.e. $G = H$, we have $[G : H] = 1$, so by the inner product formula in Theorem 28 we have that $\Phi$ sends irreducibles to irreducibles, and the weak surjectivity condition becomes usual surjectivity, so $\Phi$ is surjective and hence is a bijective isometry, so $\Phi^{-1} = \Phi^* = \Psi$.  But $\Psi$ is a map of coalgebras, so $\Phi = \Psi^{-1}$ is as well, and we obtain

\begin{corollary} For the case of usual wreath products, i.e. $G = H$, we have that $\Phi$ and $\Psi$ are mutually inverse and adjoint, positive (irreducible-to-irreducible), graded isomorphisms respecting both the algebra and coalgebra structures.  \end{corollary}

This case recovers the usual identification of $R(G, G)$ with the $|G|$-fold tensor power of the Hopf algebra of integral symmetric functions.


\begin{thebibliography}{9}

        \bibitem{Bump}
        
        D. Bump.
        \emph{Lie Groups.}
        Graduate Texts in Mathematics, Vol. 225, Springer-Verlag, New York, 2000.

        \bibitem{vanL}
	
	M. A. A. van Leeuwen.
	\emph{An Application of Hopf-Algebra Techniques to Representations of Finite Classical Groups.}
	Journal of Algebra. 06/1991; 140(1):210-246. DOI: 10.1016/0021-8693(91)90154-Z.

		
	\bibitem{Zel}
	
	A. Zelevinsky.
	\emph{Representations of Finite Classical Groups, A Hopf Algebra Approach.}
	Lecture Notes in Mathematics, Vol. 869, 1981.

\end{thebibliography}
\end{document}